\documentclass{amsart}

\usepackage{amssymb}

\newtheorem{theorem}{Theorem}[section]
\newtheorem{corollary}[theorem]{Corollary}

\newtheorem{example}[theorem]{Example}
\newtheorem{definition}[theorem]{Definition}

\begin{document}

\title{Link Algebra: A new aproach to graph theory}
\author{Alfonso Bustamante V}

\address{ Department of Mathematics, University of Chile, Las Palmeras 3425, Santiago, CHILE}

\email{abustamante@ug.uchile.cl}

\date{17 of March, 2011}

\thanks{During the development of this work, we had the initial support of Jorge Soto Andrade, Luis Arenas, Ricardo Osorio and Daniela Ruiz, which motivated us to move forward. We also received help from Boris Roa,with the idea of graph union and simple link as complete bipartite graph; Natalia Reyes with the proposal of some theorems and corrolays; and Andre Caceres in the correction and translation of the manuscript. Without this help we would not have been able to finish our work.}

\begin{abstract} 
In this paper we develop a structure called Link Algebra, in which we present a Set with two binary operations and an axiom system developed from the study of graph theory and set/antiset theory, sowing main theorems and definitions. Once introduced Link Algebra, we will show the aplication on graph theory, like defining Paths, cycles and stars. Finally, we will se an alternative axiomatizations with Multisets and ordered pairs to algebraicaly define mutli, pseudo and oriented graphs.
\end{abstract}

\maketitle

\section{Introduction}

In mathmatics, algebriac strucutres usually represent the axiomatization of known objects, hide in the dephts of a unsolved problem. The strucutre under consideration in this work, was inspired in a problem aparentley trivial: if we have a graph defined as the usual like the 2-tuple (Berge \cite{Berge})

\begin{center}
$G=(V,E),$
\end{center}

where V is the set of vertexes and E is the set of the unordered pair of vertexes (in a non-oriented graph) of edges, then we assume a graph G with his subgraph G', and we want to find the graph X in the graph G, this is

\begin{center}
$G=G'\cup X.$ 
\end{center}

If we consider only the asumptions on graph theory so far, which is bassiclay based on ZF set theory axiom system, we have no solution to the problem. If we look, instead, at the extension made in the works of Carroll \cite{Carroll}(2009) and Gatica \cite{Gatica}(2010), considering the existence of an antiset B where

\begin{center}
$A \cup B = \emptyset,$
\end{center}

with $B=\tilde{A}$ the solution would be 
$$G \cup \tilde{G'}= X,$$

where naturally emerges the concept of antigraph as a solution for a graph boolean equation system with one variable. This way of working graphs will help to bring simplest solutions to apparently complex problems.

Before we start to define our algebraic structure, we will introduce some preliminar concepts, like graphs, the definition of set union al link, and the basic of set/antiset theory.

\section{Introduction to Link Algebra}

\subsection{Graph Theory}In graph theory, the concepts of vertex and edge are essential: is the source of thousands of objects and forms, but the problem lies in one thing: it is still atached to combinatorial analisis and arithmetic formulations. Is not a propper object with an algebraic background: it is still a mathmatical bricolage of external mathematicas sub-areas. 

For many years, graphs has been consider as a mere exentricity born in the mind of Euler to solve the Koenigsberg bridge problem; in the XX century, Claude Berge made a more formal definition based on the definition by Konig, giving emphasis in the vertex and edges.In this new century, the view of graph has lightly changed: they still remains as 2-tuples only used when the ocation appears; but if we take a closer look of some basic concepts in their theory, we will be able to see the strucure that lies beneath the surface. 

In order to make a better background, lets remind the following concepts: complete bipartite graph, null graph and graph union. 

A complete bipartite graph is ussually defined as

\begin{center}
$K_{n,m}:=(V_1 \cup V_2, E),$
\end{center}

where E is the set of the pairs formed by the conection of the edges v in $V_1$ and w in $V_2$, with n and m the number of vertex contained in the set $V_1$ and $V_2$ respectively. 

The next concept is the null graph, which is defined as
$$N_n:=(V, \{\emptyset\})$$

having vertexes but no edges, where n is the number of edges. 

Finally, we have graph union which is definded as
$$G:=(V \cup V', E \cup E'),$$
where V and V' are the vertexes of two diferent graphs, being E and E' their edges. With this three concepts, we are now able to construct the main operations of Link Algebra.
\subsection{Basic Definitions} From the concepts view of graph theory, we can construct the main definitios of our work.
\begin{definition}
We define the union betwen two graphs G  and G' as
$$G \dashv \vdash G':=(V \cup V', E \cup E')$$
\end{definition}

\begin{definition}
We define the linking betwen two graphs G and G' as
\begin{enumerate}
 \item $N_n \dot{\Lambda} N'_m :=K_{n,m}$, if they are null graphs;

 \item $G \dot{\Lambda} G':=G \dashv \vdash G' \dashv \vdash K_{n,m}$, where $N_n$ and $N'_m$ are the null graphs of G and G' respectively. \footnote{If the reader is interested, use as an example $G=(\{a,b\},\{\{a,b\}\})$ and $G'=(\{c,d\},\{\{c,d\}\})$.}
\end{enumerate}
\end{definition}

Once defined this operations, we will construct the fundamental objects of graph theory: Vertexes and Edges, objects that will allow us to create the laws that will support our algebraic structure.

Back in the definition of null graph, it is not to difficult to define a single vertex, we only need n=1 and the definition is done. In the other case, the edge, we will need the definition of conection betwen two null graphs of one vertex. Hence, the concepts could be defined as follows:

\begin{definition}
We define a \textbf{Vertex} v from the single graph G as
$$v_G:=N_1,$$ where $N_1$ is the null graph of the single graph G.
\end{definition}

\begin{definition}
We define a \textbf{Edge} e from the edge graph G as
$$e_G:=N_1 \dot{\Lambda} N'_1,$$ where $N_1$  and $N'_1$ are the null graphs of the single subgraph H and H' of G respectively.
\end{definition}

Before we start defining a Conective Algebra an its laws, lets make a little introduction to set/antiset theory, for prepare the ground to the propper concept of antivertex and antiedge.

\section{A step back in to the land of sets}

\subsection{Sets and Antisets}When Zermelo and Fraenkel made the axioms of set theory, they consider sets as an unique unit in the set universe. After the appearance of the antiparticle in phisics, many mathematicians started to search a possible analogy in the land of sets. The problem was that the axiomatization made was unable to accept such antielements or antisets. In the 90's some mathematicians started to ask if there was a posibility to make an extension to accept such idea. Then in this century, works like those from Caroll or Gatica made de propper needed axiomatization. In this brief section, we will introduce the basis of set/antiset theory for using in graph theory to define antigraph,  taking the notation from the work of Carroll.

As we defined before, let be A and B sets which
$$A \cup B = \emptyset,$$
where $B= \tilde{A}$, definded as antiset of the set A.

In both works, they start by defining union in to a element view, so, union, in this form is defined for a set A and B which
$$ A \cup B:= \bigcup^{}_{} \{A,B\},$$
where A could be $A=\{a\}$ and $B=\{b\}$ where naturally

\begin{tabular}{llll}
 \\
& $A \cup B$ & $= \{a\} \cup \{b\}$ \\
&            & $= \{a,b\}$, \\

\end{tabular}
\\

which is also called axiom of fusion. If A and B are sets of subsets, the subsets interact as they were just elements, this is, there are no fusion inside fusions. Is also important to mention that is linear and the sign of the set also keeps in their elements, this is

\begin{tabular}{llll}
\\
& (a) $\widetilde{A \cup B}$ & $= \tilde{A} \cup \tilde{B}$ \\
& (b) $ \tilde{A}$           & $= \{\tilde{a}\},$
\end{tabular}
\\

in Carroll's work \cite{Carroll} , this is clearly defined, and derived naturally from the ZF axiomatic extention of the fusion and the union of elements.

with this, it is posible to define the Boolean Algebra with an extention of antielements. For that reason, the new universe set will be defined as a set with both positive and negative properties. The boolean algebraic properties remains as ussually for the operations union and intersection, now extended to antisets. 

\subsection{Consecuences in Graph theory} Using this extension in graphs, gives the naturally the concept of antigraph, this is
$$\tilde{G}:=(\tilde{V}, \tilde{E}),$$
which joined to the graph G (now extended) it result to be
$$ G \cup \tilde{G}:=(\{\emptyset\}, \{\emptyset\}),$$
where $(\{\emptyset\}, \{\emptyset\})$ is definded as the empty graph $\phi_G$.

In the same way, we will define an antivertex and an antiedge as
$$\tilde{v_G}:=\tilde{N_1},$$
$$\tilde{e_G}:=\widetilde{N_1 \dot{\Lambda} N'_1},$$

Other important consecuence, is the solution of the equation $G = G' \cup X$ seen in the introduction as $X =G \cup \tilde{G'}$.
Showed this concepts, we are finally able to define our algebraic structure.

\section{Link Algebra}

\subsection{Definition of Link Algebra} In abstract algebra, in general, the algebraic structure represents the true goal beneath the scrutine of those hide patterns in the behave of a mathematical object. Graphs has been seen as diagrams, even confussed with Graphics, missing their true nature in the world of mathmatics.

The structure, finnaly introduced in this section, has developed from the concepts in graph theory redefinded by set/antiset theory and the new definiton of vertex by using the concept of null graphs.
\begin{definition}
 A \textbf{Link Algebra} is a set G (extended to antigraphs) of vertexes v with the operations $\dashv \vdash$ and $\dot{\Lambda}$ satisfing the following axioms
\begin{enumerate}
 \item Closure
$$v \dashv \vdash w \in G,$$
$$v \dot{\Lambda} w \in G.$$
\item Idempotency
$$v \dashv \vdash v = v,$$ 
$$v \dot{\Lambda} v = v.$$
\item Asociativity \footnote{Exept for $v=w$ and $t=\tilde{v}$, to avoid the same problem in set/antiset theory. If the reader is interested, see Gatica's work.}
$$v \dashv \vdash (w \dashv \vdash t) = (v \dashv \vdash w) \dashv \vdash t,$$ 
$$v \dot{\Lambda} (w \dot{\Lambda} t) = (v \dot{\Lambda} w) \dot{\Lambda} t.$$
\item Identity element

There exist an element $\phi$, that for every elemet of G the equation
$$v \dashv \vdash \phi = v,$$ 
$$v \dot{\Lambda} \phi = v.$$
\item Inverse element

For every element v of G there exist an element $\tilde{v}$ in G that
$$v  \dashv \vdash \tilde{v} = \tilde{v} \dashv \vdash v = \phi,$$ 
$$v \dot{\Lambda} \tilde{v} = \tilde{v} \dot{\Lambda} v = \phi.$$
\item Conmutativity
$$v \dashv \vdash w = w \dashv \vdash v,$$ 
$$v \dot{\Lambda} w = w \dot{\Lambda} v.$$
\item Distributivity
\begin{enumerate}
 \item $v \dot{\Lambda} (w \dashv \vdash t) = (v \dot{\Lambda} w) \dashv \vdash (v \dot{\Lambda} t).$
 \item $(w \dashv \vdash t) \dot{\Lambda} v = (w \dot{\Lambda} v) \dashv \vdash (t \dot{\Lambda} v).$
\end{enumerate}
\item Conectivity

For every equation $\Gamma$ and $\Gamma'$ of G, where
$$\Gamma = V \dashv \vdash E,$$ 
being V and E defined as 
\begin{enumerate}
 \item $V= v_1 \dashv \vdash \ldots \dashv \vdash v_p$
 \item $E= w_1 \dot{\Lambda} t_1 \dashv \vdash \ldots \dashv \vdash w_q \dot{\Lambda} t_q,$
 \item $N_n = V \dashv \vdash w_1 \dashv \vdash t_1 \dashv \vdash \ldots \dashv \vdash w_q \dashv \vdash t_q $
\end{enumerate}

with $N_n$ (where $n=p  + 2q$) and $N'_m$ their vertex equations
$$\Gamma \dot{\Lambda} \Gamma':= \Gamma \dashv \vdash \Gamma' \dashv \vdash N_n \dot{\Lambda} N'_m.$$
\end{enumerate}

\end{definition}
 
Some of the axioms within the structure, bring the open posiblity to construct some important objects in graph theory, such as Paths, Cycles, Stars and complete graphs; but before, we will look at the fundamental properties of the Link Algebra to continue with the consrtuction of main objects such K, P, C and S equations.

\subsection{Theorems of Link Algebra} The following theorems, except few, have proofs using just the basic of group theory. If the reader is interested, could see this proofs in the work of Kaufman and Percigout \cite{Kauffman}, in the sections of groups. 

\begin{theorem}
 if v,w and t $\in$ G and $v \dot{\Lambda} t = w \dot{\Lambda} t$, then $v=w$
\end{theorem}

\begin{proof}
One way to solve the equation $v \dot{\Lambda} t = w \dot{\Lambda} t$ is by linking $\tilde{t}$ and using the axiom (8), this is

\begin{tabular}{llll}
& $v \dot{\Lambda} t$ & $= w \dot{\Lambda} t$ & {} \\
& $(v \dot{\Lambda} t) \dot{\Lambda} \tilde{t}$ & $= (w \dot{\Lambda} t) \dot{\Lambda} \tilde{t}$ \\
& $v \dot{\Lambda} t \dashv \vdash \tilde{t} \dashv \vdash (v \dashv \vdash t) \dot{\Lambda} \tilde{t}$ & $=w \dot{\Lambda} t \dashv \vdash \tilde{t} \dashv \vdash (w \dot{\Lambda} t) \dot{\Lambda} \tilde{t}$ \\
& $v \dot{\Lambda} t \dashv \vdash \tilde{t} \dashv \vdash (v \dot{\Lambda} \tilde{t}) \dashv \vdash (t \dot{\Lambda} \tilde{t})$ & $=w \dot{\Lambda} t \dashv \vdash \tilde{t} \dashv \vdash (w \dot{\Lambda} \tilde{t}) \dashv \vdash (t \dot{\Lambda} \tilde{t})$ \\
& $v \dot{\Lambda} t \dashv \vdash \tilde{t} \dashv \vdash (v \dot{\Lambda} \tilde{t}) \dashv \vdash \phi$ & $=w \dot{\Lambda} t \dashv \vdash \tilde{t} \dashv \vdash (w \dot{\Lambda} \tilde{t}) \dashv \vdash \phi$ \\
& $v \dot{\Lambda} t \dashv \vdash \tilde{t} \dashv \vdash (v \dot{\Lambda} \tilde{t})$ & $=w \dot{\Lambda} t \dashv \vdash \tilde{t} \dashv \vdash (w \dot{\Lambda} \tilde{t})$ \\
& $v \dot{\Lambda} t \dashv \vdash \tilde{t} \dot{\Lambda} (\phi \dashv \vdash v)$ & $=w \dot{\Lambda} t \dashv \vdash \tilde{t} \dot{\Lambda} (\phi \dashv \vdash w)$ \\
& $v \dot{\Lambda} t \dashv \vdash \tilde{t} \dot{\Lambda} v$ & $=w \dot{\Lambda} t \dashv \vdash \tilde{t} \dot{\Lambda} w$ \\
& $v \dot{\Lambda} (t \dashv \vdash \tilde{t})$ & $=w \dot{\Lambda} (t \dashv \vdash \tilde{t})$ \\
& $v \dot{\Lambda} \phi$ & $w \dot{\Lambda} \phi$ \\
& $v$ & $=w$ \\
\end{tabular}
\\
The other way, is by using the axiom (3), (5) and (4).
\end{proof}

\begin{theorem}
 For every v in G, $v \dot{\Lambda} \phi = \phi \dot{\Lambda} v$
\end{theorem}

\begin{theorem}
 A link Algebra, has a single neutral element.
\end{theorem}

\begin{proof}
Lets supose that there is a $\phi'$ that $\phi \dot{\Lambda} \phi' = \phi$ and $\phi' \dot{\Lambda} \phi = \phi',$ occurs.
Using the axiom (6), the consequence is immediate. The same occurs for $\dashv \vdash$.
\end{proof}

\begin{theorem}
 For every v of G, $\tilde{v} \dot{\Lambda} v = v \dot{\Lambda} \tilde{v}$
\end{theorem}

\begin{theorem}\label{theo alg rght}
 if v,w and t $\in$ and $t \dot{\Lambda} v = t \dot{\Lambda} w$, then $v=w$
\end{theorem}

\begin{theorem}
 Every element of a link algebra have a single inverse element.
\end{theorem}

\begin{theorem}
 If v $\in$ G, then $\widetilde{\tilde{v}} = v$
\end{theorem}

\begin{proof}
We want to know if there is some $\widetilde{\tilde{v}}$ that
$$\tilde{v} \dashv \vdash \widetilde{\tilde{v}} = \phi$$
 lets take the axiom (5) for $\dashv \vdash$
$$\tilde{v} \dashv \vdash v = \phi,$$
using the theorem \ref{theo alg rght}, we proof that $\widetilde{\tilde{v}} = v$
\end{proof}

\begin{theorem}
 If v and w $\in$ G, exists single x and y of G that $v \dot{\Lambda} x = B$ and $y \dot{\Lambda} v = B$
\end{theorem}

\begin{theorem}
 (G, $\dashv \vdash$) obey group theorems. 
\end{theorem}

\subsection{Objects in Link Algebra} In Link algebra, same as graph theory, there are fundamental objects like Paths, Cycles, Stars and Complete Conective forms. In this section we will present the definition of those objects.

\begin{definition}\label{def path}
$\forall$ $v_1, \dots, v_n \in$ G, an equation is called \textbf{Path} if
$$ P_n = v_1 \dot{\Lambda} v_2 \dashv \vdash \dots \dashv \vdash v_{n-1} \dot{\Lambda} v_n$$
\end{definition}

\begin{definition} 
$\forall$ $v_1, \dots, v_n \in$ G, an equation is called \textbf{Cycle} if
$$C_n = P_n \dashv \vdash v_1 \dot{\Lambda} v_n$$
\end{definition}

\begin{definition}\label{def star}
$\forall$ $v_1, \dots, v_n$ and w $\in$ G, an equation is called \textbf{Star} if
$$ S_n = w \dot{\Lambda} (v_1 \dashv \vdash \dots \dashv \vdash v_n)$$
\end{definition}

\begin{definition}
$\forall$ $v_1, \dots, v_n \in$ G, an equation is called \textbf{Complete Conective Form} if
$$K_n = v_1 \dot{\Lambda} \dots \dot{\Lambda} v_n.$$

\end{definition}

\begin{definition}
$\forall$ v and w $\in$ G, an Edge is definded as
$$\bar{\varepsilon} = v \dot{\Lambda} w.$$
\end{definition}

\begin{definition}
$\forall$ $v_1, \dots, v_n \in$ G, an equation is called Null Form if
$$N_n = v_1 \dashv \vdash \dots \dashv \vdash v_n.$$
\end{definition}

\subsection{Aplication on Graph Theory} As we said in the introduction, we will show how Link Algebra is useful on Graph Theory. Therefore, we will give some main theorems derived from the figures in the earlier section. Some will be demonstrated, the rest are easy to proof from the other theorems.

\begin{theorem}
if $V=v_1 \dashv \vdash \dots \dashv \vdash v_n$, then $K_n = V \dot{\Lambda} V.$
\end{theorem}

\begin{proof}
 Using the definition of $K_n$ and using recursively the axiom (8)

\begin{tabular}{llll}
& $K_n$ & $= v_1 \dot{\Lambda} \dots \dot{\Lambda} v_n$ & {}\\
& {} & $= v_1 \dashv \vdash \dots \dashv \vdash v_n \dashv \vdash v_{n-1} \dot{\Lambda} v_n \dashv \vdash \dots \dashv \vdash v_1 \dot{\Lambda} (v_2 \dashv \vdash \dots \dashv \vdash v_n)$  \\
& {} & $=v_1 \dot{\Lambda} v_1 \dashv \vdash \dots \dashv \vdash v_n \dot{\Lambda} v_n \dashv \vdash S_1 \dashv \vdash \dot \dashv \vdash S_n$  \\
& {} & $=v_1 \dot{\Lambda} v_1 \dashv \vdash \dots \dashv \vdash v_n \dot{\Lambda} v_n \dashv \vdash S$  \\
& {} & $=v_1 \dot{\Lambda} v_1 \dashv \vdash \dots \dashv \vdash v_n \dot{\Lambda} v_n \dashv \vdash S \dashv \vdash S$  \\
& {} & $=v_1 \dot{\Lambda} (v_1 \dashv \vdash \dots \dashv \vdash v_n) \dashv \vdash \dots \dashv \vdash v_n \dot{\Lambda} (v_1 \dashv \vdash \dots \dashv \vdash v_n \dot{\Lambda} v_n)$ \\
& {} & $=V \dot{\Lambda} V$  \\
\end{tabular}
\\
Where $S=S_1 \dashv \vdash \dots \dashv \vdash S_n$ And the proof is done.
\end{proof}

\begin{theorem}\label{star v theo}
 if $V= v_1 \dashv \vdash \dots \dashv \vdash v_n$, then $S_n \dashv \vdash V = S_n$
\end{theorem}

\begin{proof}
 using the definition of star

\begin{tabular}{llll}
& $S_n \dashv \vdash V$ & $= w \dot{\Lambda} (v_1 \dashv \vdash \dots \dashv \vdash v_n) \dashv \vdash V$ & {} \\
& {} & $=w \dot{\Lambda} V \dashv \vdash V$ & (using the hipotesis) \\
& {} & $=w \dot{\Lambda} V \dashv \vdash V \dot{\Lambda} \phi$ & (using (4)) \\
& {} & $=w \dot{\Lambda} (V \dashv \vdash \phi)$ & (using (7)) \\
& {} & $=S_n$ & (using (4) and the def \ref{def star})
\end{tabular}
\\
which is, as a matter of fact, a Star.
\end{proof}

\begin{theorem}
 $P_n$ is inductive.
\end{theorem}

\begin{proof}
Using Peano Axioms, the case n=1 is trivial. Now, we need to prove that $P_n$ implies $P_{n+1}$.

\begin{tabular}{llll}
& $P_{n+1}$ & $=v_1 \dot{\Lambda} v_2 \dashv \vdash \dots \dashv \vdash v_{n-1} \dot{\Lambda} v_n \dashv \vdash v_n \dot{\Lambda} v_{n+1}$ & (using def \ref{def path}) \\
& {} & $=P_n \dashv \vdash v_n \dot{\Lambda} v_{n+1}$ & (using def \ref{def path} again) \\
\end{tabular}
\\
so, $P_n$ is inductive.
\end{proof}

\begin{corollary}\label{cor cy}
 $C_n$ is inductive.
\end{corollary}

\begin{proof}
 the case n=1, as in the theorem, is trivial. We need to prove that $C_n$ implies $C_{n+1}$.

\begin{tabular}{llll}
& $C_{n+1}$ & $= P_{n+1} \dashv \vdash v_1 \dot{\Lambda} v_{n+1}$\\
& {} & $= P_n \dashv \vdash v_n \dot{\Lambda} v_{n+1} \dashv \vdash v_1 \dot{\Lambda} v_{n+1}$\\
& {} & $= P_n \dashv \vdash R \dashv \vdash v_1 \dot{\Lambda} v_n \dashv \vdash \widetilde{v_1 \dot{\Lambda} v_n}$ \\
& {} & $= P_n \dashv \vdash v_1 \dot{\Lambda} v_n \dashv \vdash R \dashv \vdash \widetilde{v_1 \dot{\Lambda} v_n}$ \\
& {} & $=C_n \dashv \vdash R \dashv \vdash \widetilde{v_1 \dot{\Lambda} v_n},$ \\
\end{tabular}
\\
calling $R=v_n \dot{\Lambda} v_{n+1} \dashv \vdash v_1 \dot{\Lambda} v_{n+1}$, $C_n$ is inductive.
\end{proof}

\begin{corollary}
$S_n$ is inductive.
\end{corollary}

\begin{proof}
 Same as corollary \ref{cor cy} .
\end{proof}

\begin{theorem}
if $S_n= S_{n-1} \dashv \vdash \bar{\varepsilon}$ and $S'_m=S'_{m-1} \dashv \vdash \bar{\varepsilon}$, then 
$$S_n \dashv \vdash S'_m \dashv \vdash \tilde{\bar{\varepsilon}}=S_{n-1} \dashv \vdash S'_{m-1}.$$
\end{theorem}

\begin{proof}
first we use the definitions

\begin{tabular}{llll}
& $S_n \dashv \vdash S'_m \dashv \vdash \tilde{\bar{\varepsilon}}$ & $=(S_{n-1} \dashv \vdash \bar{\varepsilon}) \dashv \vdash (S'_{m-1} \dashv \vdash \bar{\varepsilon}) \dashv \vdash \tilde{\bar{\varepsilon}}$ \\
& {} & $=(S_{n-1} \dashv \vdash S'_{m-1}) \dashv \vdash (\bar{\varepsilon} \dashv \vdash \bar{\varepsilon}) \dashv \vdash \tilde{\bar{\varepsilon}}$ \\
& {} & $=S_{n-1} \dashv \vdash S_{m-1} \dashv \vdash (\bar{\varepsilon} \dashv \vdash \tilde{\bar{\varepsilon}})$ \\
& {} & $=S_{n-1} \dashv \vdash S_{m-1} \dashv \vdash \phi$ \\
& {} & $=S_{n-1} \dashv \vdash S_{m-1}$ \\
\end{tabular}
\\
then, we only use idempotency on edges and the axiom of inverse element.
\end{proof}

\begin{theorem}
If $V=V' \dashv \vdash V''$ where $V'=v_1 \dashv \vdash \dots \dashv \vdash v_j$, $V''=v_j \dashv \vdash \dots \dashv \vdash v_n$ where $j \in \{1, \dots, n \}$, then $S_n=S'_j \dashv \vdash S''_{n-j+1}$.
\end{theorem}

\begin{proof}
Using the axiom of distributivity and asociativity, and the hipothesis, the proof is obvious. 
\end{proof}

Before we present the next definition, lets define the following concepts

\begin{definition}
If $f(v_1, \dots, v_n)=\Gamma$ is a conective form where $f:G^n \rightarrow G$, we will call Conective Form Complement to
$$\Gamma^c= K_n \dashv \vdash \tilde{\Gamma}$$
\end{definition}

\begin{definition}\label{def cf}
If $\Gamma$ is a Conective Form, we will call a Star Composed Form to
$$\Gamma_s= N_n \dot{\Lambda} N_n \dashv \vdash \widetilde{\Gamma^c}$$
\end{definition}

\begin{example}
Lets say that we have a conective form 
$$\Gamma = v_1 \dot{\Lambda} v_2 \dot{\Lambda} v_3 \dashv \vdash v_3 \dot{\Lambda} v_4.$$
The complement of $\Gamma$, using the definition will be:

\begin{tabular}{llll}
& $\Gamma^c$ & $= N_4 \dot{\Lambda} N_4 \dashv \vdash \widetilde{(v_1 \dot{\Lambda} v_2 \dot{\Lambda} v_3 \dashv \vdash v_3 \dot{\Lambda} v_4)}$ \\
& {} & $=v_1 \dot{\Lambda} (v_2 \dashv \vdash v_3 \dashv \vdash v_4 \dashv \vdash v_5) \dashv \vdash \dots \dashv \vdash v_4 \dot{\Lambda} v_5 \dashv \vdash \widetilde{v_1 \dot{\Lambda} v_2 \dot{\Lambda} v_3} \dashv \vdash \widetilde{v_3 \dot{\Lambda} v_4}$ \\
& {} & $=v_1 \dot{\Lambda} v_4 \dashv \vdash v_4 \dot{\Lambda} v_2.$ \\
\end{tabular}
\\
Hence, the SCF:

\begin{tabular}{llll}
& $\Gamma_s$ & $= N_4 \dot{\Lambda} N_4 \dashv \vdash \widetilde{(v_1 \dot{\Lambda} v_4 \dashv \vdash v_4 \dot{\Lambda} v_2)}$ \\
& {} & $=(v_1 \dashv \vdash v_2 \dashv \vdash v_3 \dashv \vdash v_4) \dot{\Lambda} (v_1 \dashv \vdash v_2 \dashv \vdash v_3 \dashv \vdash v_4) \dashv \vdash \widetilde{v_1 \dot{\Lambda} v_4} \dashv \vdash \widetilde{v_4 \dot{\Lambda} v_2}$ \\
& {} & $=v_1 \dot{\Lambda} (v_2 \dashv \vdash v_3) \dashv \vdash v_2 \dot{\Lambda} (v_1 \dashv \vdash v_3) \dashv \vdash v_3 \dot{\Lambda} (v_1 \dashv \vdash v_2 \dashv \vdash v_4) \dashv \vdash v_4 \dot{\Lambda} v_3.$ \\
\end{tabular}
\\
\end{example}

\begin{definition}\label{theo card star}
If $V=v_1 \dashv \vdash \dots \dashv \vdash v_n$, with $S_n$ defined as usual, for $f:G \rightarrow G$, $f(S_n) = S'_m$, if
$$card(V) = card(V'),$$
With card, defined as the cardinality of the star.
\end{definition}

\begin{definition}
The graphs $\Gamma$ and $\Gamma'$ will be isomorphic (ordered form maximum to minimum star by cardinality) if there is some $f:G \rightarrow G$, lineal to $\dashv \vdash$ that
$$f(\Gamma_s)= \Gamma'_s.$$
\end{definition}

\begin{example}
Lets supose we have the graphs $\Gamma = v_1 \dot{\Lambda} v_2 \dot{\Lambda} v_3 \dashv \vdash v_3 \dot{\Lambda} v_4$ and $\Gamma'= w_1 \dot{\Lambda} (w_2 \dashv \vdash w_3 \dashv \vdash w_4) \dashv \vdash w_2 \dot{\Lambda} w_4$.
Applying the definition \ref{def cf} in both, we have
$$\Gamma_s= v_1 \dot{\Lambda} (v_2 \dashv \vdash v_3) \dashv \vdash v_2 \dot{\Lambda} (v_1 \dashv \vdash v_3) \dashv \vdash v_3 \dot{\Lambda} (v_1 \dashv \vdash v_2 \dashv \vdash v_4) \dashv \vdash v_4 \dot{\Lambda} v_3$$
$$\Gamma'_s= w_1 \dot{\Lambda} (w_2 \dashv \vdash w_3 \dashv \vdash w_4) \dashv \vdash w_2 \dot{\Lambda} (w_3 \dashv \vdash w_4) \dashv \vdash w_3 \dot{\Lambda} w_1 \dashv \vdash w_4 \dot{\Lambda} (w_1 \dashv \vdash w_2),$$
as we see, ordering by maximum to minimum star and applying a function to $\Gamma$ (lineal to $\dashv \vdash$),the stars have equal cardinality, so, we acctually have $f(\Gamma_s) = \Gamma'_s$.
\end{example}

\begin{theorem}
if $\Gamma^c$ is the complement of $\Gamma$ with n vertex both, then 
$$\Gamma^c \dashv \vdash \Gamma= K_n.$$
\end{theorem}

\begin{proof}
Using the defintion of complementary graph

\begin{tabular}{llll}
&  $\Gamma^c \dashv \vdash \Gamma$ & $=(K_n \dashv \vdash \tilde{\Gamma}) \dashv \vdash \Gamma$ \\
& {} & $=K_n \dashv \vdash (\tilde{\Gamma} \dashv \vdash \Gamma)$ \\
& {} & $=K_n \dashv \vdash \phi$ \\
& {} & $=K_n.$ \\
\end{tabular}
\\
the graph and the complement make the complete graph.
\end{proof}

\section{A brief view of applications on multi, pseudo an oriented Graphs}

\subsection{A reminder of the definitions}\label{subsect} In many mathematical bibliography, graphs seems to be confused with other kind of objects like multi graphs, pseudo graphs and oriented graps. The truth is, there are conceptual differences, of which we are going to make a short remind to the reader.
\begin{definition}
A graph $G=(V,E)$ is called \textbf{multigraph} if some edges could be equal; between two nodes, can exist many edges.
\end{definition}

\begin{definition}
A graph is called \textbf{pseudograph} if there are edges with the form $\{a,a\}$, but no multiedges of the canonical form.
\end{definition}

\begin{definition}
A graph with the form $G=(V,E)$ is called \textbf{oriented} if the edges of E have the form $e=(v,w)$, which is an ordered pair of vertexes.
\end{definition}

\subsection{Arithmetical Graphs} Wolfram, in an article referred at Pseudographs \cite{Wolfram}, referred as a graph with loops an multiedges. For some authors, pseudographs and graphs are deeply connected by the idea of multiset. But the problem was that the definition of those graphs only consider multiedges but no multivertexes, by that reason there is no posibility to define precisely the operation of union.

In 2003, Wildberger \cite{Wildberger} presented a work proposing an alternative form to treat multisets by using linear notation. In that work, he propose a third operation on sets: the sum, were elements of two sets A and B just add in even if elements of both repeat. Let's give a simple example. Let's define the multisets A and B as
$$A= \{a,a,b\}$$
$$B=\{b,b,c\}$$
Then
$$A+B=\{a,a,b,b,b,c\}$$

If we add antisets, we will have something very close to the arithmetical structure of $\mathbb{Z}$ with the operation sum, defined as an abelian group. In the work of Carroll, there is an extension of the concept of negative sets to negative numbers that acctually shows how group theory can be constructed using the new extension to ZF axiomatic.

Multisets are linear, this is, for an external operation product in $\mathbb{Z}$

\begin{enumerate}
 \item $(n+m)S=nS+mS$
 \item $(nm)S=n(mS)$
 \item $n(S+S')=nS+nS'$
 \item $1S=S$
\end{enumerate}

If the reader is interested, a deepest analisis of this structure of multisets is avaible on the mencioned work. 

Following this concept closer, we will observe that if we have a graph $G=(V,E)$ the definition of graph sum extend not only to the edges but to the vertexes.Now, if we consider the definition on Wildbergers work, then we will have the folowing definition.

\begin{definition}
Lets say that an \textbf{arithmetical graph} is such graph where V and E are multisets, with V multiset or vertexes and E multiset of the non ordered pairs of vertexes or edges.
\end{definition}

Now, lets continue with the change of our previous algebra.

\begin{definition}
A set $\Pi$ with the operations $\dashv \mid \vdash$ and $\mathring{\Lambda}$ is called \textbf{Arithmetical link}
if is a link algebra, but where axiom (4) is $v \dashv \mid \vdash \phi=v, v \mathring{\Lambda} \phi = \phi$, and both operations are not idempotent, with an external operation (multiplication) on $\mathbb{Z}$ (is linear).
\end{definition}

the theorems remain equal, the only variation of this algebra is the introduction of the concept fo Vertex and Edge Engrosure , which includes the properties of numbers in graphs.

\begin{definition}
we will call the relation $\dot{\prec}$ \textbf{Engrosure}, where $\forall n,m \in \mathbb{Z}$ and v $\in \Pi$ we have $n(v) \dot{\prec} m(v)$ if $n < m$.
\end{definition}

is not too difficult to see that if $(\Pi, \dashv \mid \vdash)$ is a group, $(\Pi, \dot{\prec})$ is an equivalence relation. The reader can easily prove the reflexivity, symmetry and transitivity.

\begin{definition}
we will call \textbf{Inverse vertex} to
$$(-1)v := \tilde{v}.$$
\end{definition}

\begin{definition}
We will call a \textbf{Loop} to
$$v^\sigma := v \mathring{\Lambda} v.$$
\end{definition}

there is a theorem we want to add to this structure, probably it is obvious to demonstrate, but it is worth to be written.

\begin{theorem}
For an edge $e=v \mathring{\Lambda} w$ with $n \in \mathbb{Z}$
$$ne=(nv) \mathring{\Lambda} w = v \mathring{\Lambda} (nw).$$
\end{theorem}

\begin{proof}
we have

\begin{tabular}{llll}
& $n(v \mathring{\Lambda} w)$ & $= v \mathring{\Lambda} w \dashv \mid \vdash \dots \dashv \mid \vdash v \mathring{\Lambda} w$ \\
& {} & $=v \mathring{\Lambda} (w \dashv \mid \vdash \dots \dashv \mid \vdash w)$ \\
& {} & $=v \mathring{\Lambda} (nv)$ \\ 
\end{tabular}
\\
the other case is analogous.
\end{proof}

\begin{example}
Let's say we have the multigraph 
$$\Pi = 2(v_1 \mathring{\Lambda} v_2) \dashv \mid \vdash 2(v_2 \mathring{\Lambda} v_3) \dashv \mid \vdash v_4 \mathring{\Lambda} (v_1\dashv \mid \vdash v_2 \dashv \mid \vdash v_3).$$
Using the linearity he have:
$$\Pi = 2(v_1 \mathring{\Lambda} v_2 \dashv \mid \vdash v_2 \mathring{\Lambda} v_3) \dashv \mid \vdash v_4 \mathring{\Lambda} (v_1\dashv \mid \vdash v_2 \dashv \mid \vdash v_3),$$
Using the defintion of Path and star:
$$\Pi = 2(P_3) \dashv \mid \vdash S_3.$$
This example is an homage to Bernhard Euler, to his famous solution of the Koenisberg bridge problem presented on his work of 1736 Solutio problematis ad geometriam situs pertinentis.
\end{example}

\subsection{Oriented and mixed graphs} This subsection will be briefest than others, the reason is that the definition introduced is short.
As we said in the subsection \ref{subsect}, an oriented graph is defined as a graph $G=(V,E)$ where E is an ordered pair of vertexes. Therefore, the only thing we will define before to show the definition of oriented link algebra, is the following

\begin{definition}
Let's have (x,y) with Kuratowski's defintion
$$(x,y):=\{\{x\},\{x,y\}\},$$
We will call \textbf{Oriented Twist} to
$$(x,y)^\circlearrowleft=(y,x).$$
\end{definition}

A more precise definition with Set/antiset theory would be
$$(x,y)^\circlearrowleft=(x,y) \cup \{\tilde{\{x\}},\{y\}\},$$
where $\{x\}=A$ and $\tilde{\{x\}}= \tilde{A}$.

Now, lets define An Oriented Link Algebra

\begin{definition}
An \textbf{Oriented Link Algebra} is a set $\varGamma$ with operations $\dashv \vdash$, $\vec{\Lambda}$ and ${}^\circlearrowleft$ (unitary), that satisfies the axioms of a Link Algebra except commutativity for the operation $\vec{\Lambda}$ satisfying instead
\begin{enumerate}
 \item $v \vec{\Lambda} w = w \overleftarrow{\Lambda} v$
 \item $(v \vec{\Lambda} w)^\circlearrowleft=v \overleftarrow{\Lambda} w$
\end{enumerate}

Basically, it works same as a link algebra, with the same theorems and definitions.
\end{definition}

For a mixed graphs the definition is as follow

\begin{definition}
An \textbf{Mixed Link Algebra} is a set $\varGamma$ with the operations $\dashv \vdash$, $\dot{\Lambda}$, $\vec{\Lambda}$, ${}^\circlearrowleft$ where

\begin{enumerate}
 \item ($\varGamma$, $\dashv \vdash$, $\dot{\Lambda}$) is a Link Algebra.
 \item ($\varGamma$, $\dashv \vdash$, $\vec{\Lambda}$, ${}^\circlearrowleft$) is a Oriented Link Algebra.
\end{enumerate}

\end{definition}

To finish the exposition of this concept we will give two last examples.

\begin{example}
Let's have the oriented graph
$$\Gamma = w_1 \vec{\Lambda} v_1 \dashv \vdash v_1 \vec{\Lambda} v_2 \dashv \vdash v_2 \vec{\Lambda} v_3 \dashv \vdash v_1 \overleftarrow{\Lambda} v_4 \dashv \vdash v_4 \overleftarrow{\Lambda} v_3$$
using the concept of antisimetric conmutativity
$$\Gamma = w_1 \vec{\Lambda} v_1 \dashv \vdash v_1 \vec{\Lambda} v_2 \dashv \vdash v_2 \vec{\Lambda} v_3 \dashv \vdash v_4 \vec{\Lambda} v_1 \dashv \vdash v_3 \vec{\Lambda} v_4$$
\end{example}

\begin{example}
Let's have the mixed graph
$$\Gamma = v_1 \vec{\Lambda} v_2 \dashv \vdash v_2 \vec{\Lambda} v_1 \dashv \vdash v_2 \dot{\Lambda} v_3.$$
Where we can identify a circular vinculation and a simple vinculation.
\end{example}

\subsection{As an Epilogue} For almost 200 years, graphs had a minor status in the world of mathmatics, being just the shadow of combinatorics, being only consulted when combinatorial arithmetic fails. George Boole, the father of modern logic, once said that in a near future, mathematic will be not mathmatics of numbers, but of abstract objects beyond numbers, opening the view to a new world where there is no center in the mathematical universe: just objects, operations, lattices and who knows what other concepts to be defined and discover. 

Boole as same as Sassure (father of modern semiotics), observe that Mathematical simbols \textbf{are} if there is a meaning that support them, an essence that allow to work the abstraction as it were concrete, like the Demiurgos of Plato's Timeo, who commiserating of our material universe, decided to printed material copies from the universe of ideas. \textbf{Demiurgos} means \textbf{Artisan} in greek, and that is the essence of the work of a mathematitian: be the artisan of a never ending sculpture, the diachronically undefined but synchronously defined mathematical Rodin's Dante: thinking of the \textit{commedia della ragione}  in the top of the gate of hell. 

Le Graphe est mort, vive le Graphe !


\begin{thebibliography}{}

\bibitem{Gatica} I.Gatica, $\sigma$-Set theory: Introduction to the concepts of $\sigma$-antielement, $\sigma$-antiset and Integer Space, (2009) preprint.

\bibitem{Carroll} M. L. Carroll, Sets and Antisets, (2009) prepint.

\bibitem{Kauffman} Kaufman and Pircegout, Curso de matematicas nuevas, Cecsa (1970).

\bibitem{Berge} C.Berge, Teoria de las redes y sus aplicaciones, Cecsa (1962). 

\bibitem{Wolfram} S.Wolfram, Math World, http://mathworld.wolfram.com/Pseudograph.html

\bibitem{Wildberger} N.J.Wildberger, A new look at mutisets. (2003) preprint.


\end{thebibliography}
\end{document}